\RequirePackage{fix-cm}
\documentclass[smallextended,numbook,runningheads]{svjour3}       
\smartqed  
\usepackage{graphicx}
\usepackage{amsmath,amsfonts,amsbsy,amstext,amsopn,amssymb}
\usepackage{longtable}
\usepackage{theorem}
\usepackage{cite}
\usepackage{booktabs}
\usepackage{anysize,amsmath,indentfirst,geometry,xcolor}
\usepackage[colorlinks,citecolor=cyan,linkcolor=blue]{hyperref}

\smartqed  
\usepackage{amsfonts,anysize,amsmath,indentfirst,geometry,xcolor}
\usepackage[colorlinks,citecolor=cyan,linkcolor=blue]{hyperref}
\usepackage{amssymb}
\usepackage{mathrsfs}
\usepackage[figuresleft]{rotating}
\usepackage{graphicx}
\usepackage{amsbsy}
\usepackage{eqnarray}
\usepackage{booktabs}
\usepackage{multirow}
\usepackage{epstopdf}
\usepackage{array}
\usepackage{algorithm}
\usepackage{algorithmic}
\usepackage{graphicx}
\usepackage{wasysym}
\usepackage[misc]{ifsym}
\usepackage{cite}
\usepackage{epstopdf}

\usepackage{latexsym}
\usepackage{amssymb}
\usepackage{graphicx}
\usepackage{algorithm}

\newcommand {\eq} [1] {\begin{equation}\label{#1}}
\newcommand {\en} {\end{equation}}
\newcommand {\cA}       {{\cal A}}
\newcommand {\cB}       {{\cal B}}

\newcommand {\cP}       {{\cal P}}

\newcommand {\cX}       {{\cal X}}
\newcommand {\cY}       {{\cal Y}}

\newcommand {\R}        {{\mathbb R}}

\newcommand {\mat}      [1] {\left[\begin{array}{#1}}
\newcommand {\rix}          {\end{array}\right]}

\newcommand {\rank}     {\mathop{\rm rank}\nolimits}

\newcommand {\trace}    {\mathop{\rm trace}\nolimits}

 \catcode`@=11
 \font\tenex=cmex10 
 \newdimen\p@renwd
 \setbox0=\hbox{\tenex B} \p@renwd=\wd0 
 \def\bmat#1{\begingroup \m@th
   \setbox\z@\vbox{\def\cr{\crcr\noalign{\kern2\p@\global\let\cr\endline}}%
     \ialign{$##$\hfil\kern2\p@\kern\p@renwd&\thinspace\hfil$##$\hfil
       &&\quad\hfil$##$\hfil\crcr
       \omit\strut\hfil\crcr\noalign{\kern-\baselineskip}%
       #1\crcr\omit\strut\cr}}%
   \setbox\tw@\vbox{\unvcopy\z@\global\setbox\@ne\lastbox}%
   \setbox\tw@\hbox{\unhbox\@ne\unskip\global\setbox\@ne\lastbox}%
   \setbox\tw@\hbox{$\kern\wd\@ne\kern-\p@renwd\left[\kern-\wd\@ne
     \global\setbox\@ne\vbox{\box\@ne\kern2\p@}%
     \vcenter{\kern-\ht\@ne\unvbox\z@\kern-\baselineskip}\,\right]$}%
   \null\;\vbox{\kern\ht\@ne\box\tw@}\endgroup}
 \catcode`@=12
\newcommand{\vect}{{\sf{vec}}}

\newcommand{\rt}{{\top }}

\def\e{\varepsilon}
\def\Oh{{\mathcal O}}

\def\bfe{{\mathbf e}}

\def\bfe{{\mathbf e}}
\def\bfK{{\mathcal K}}
\def\bfP{{\mathcal P}}
\def\bfH{{\mathcal M}}
\def\bfJ{{\mathcal J}}

\def\bu{{\mathbf u }}
\def\bv{{\mathbf v }}
\def\bu{{\mathbf u }}
\def\bz{{\mathbf z }}
\def\bw{{\mathbf w }}

\def\bfA{{\mathcal A}}
\def\bfK{{\mathcal K}}

\def\bfP{{\mathcal P}}
\def\bfH{{ M}}
\def\bfJ{{N}}

\def\sd{{\sf d}}

\newcommand{\call}{{\mathbf{L}}}

\def\rel{{\rm rel}}
\def\u{{\rm U}}
\usepackage{lscape}

\begin{document}

\title{Mixed and componentwise condition numbers for a linear function of the solution of the linear least squares problem with equality constrains
}

\titlerunning{Mixed and componentwise condition numbers for LSE}        

\author{Huai-An Diao
}


\institute{H.A. Diao \at
              School of Mathematics and Statistics, Northeast
Normal University,\\
 No. 5268 Renmin Street, Chang Chun 130024, P.R.
of China. \\
              \email{hadiao@nenu.edu.cn, hadiao78@yahoo.com }
}

\date{Received: date / Accepted: date}

\maketitle

\begin{abstract}
In this paper, we consider the mixed and componentwise condition numbers for a linear function of the solution to the linear least squares problem with equality constrains (LSE). We derive the explicit expressions of the mixed and componentwise condition numbers through the dual techniques. The sharp upper bounds for the derived  mixed and componentwise condition numbers are obtained, which can be estimated efficiently by means of the classical Hager-Higham algorithm for estimating matrix one-norm   during using the generalized QR factorization method for solving LSE. The numerical examples show that the derived condition numbers can give sharp perturbation bounds, on the other hand normwise condition numbers can severely overestimate the relative errors because normwise condition numbers ignore the data sparsity and scaling.
\end{abstract}

\section{Introduction}
\setcounter{equation}{0}

The least squares  problem with equality constrains (LSE) has  the following form:
\begin{equation}\label{eq:lsle}
 \mbox{LSE}:\qquad\min_{x\in \R^n}\|Ax-b\|_2 \mbox{ subject to } Cx=d,
\end{equation}
where $A \in \R^{m \times n}, \, C \in \R^{p \times n}, \, b\in
\R^m \mbox{ and } d \in \R^p.$ The rank conditions \cite{Bjorck96}
\begin{equation}\label{eq:rank condition}
\rank(C)=p \mbox{ and } \rank\left(\begin{matrix}A \cr
C\end{matrix}\right)=n
\end{equation}
guarantee the existence of the unique solution of LSE \cite{Bjorck96,CoxHigham99}
$$
x=\bfK b+C_A^\dagger d,
$$ where
 $$
 \bfK=(A \bfP)^\dagger,\quad \bfP=I_n-C^\dagger C,\quad C_A^\dagger=\left(I_n-\bfK A\right)C^\dagger,
 $$
and $B^\dagger$ is the Moore-Penrose inverse of $B$ \cite{Bjorck96}. Under the rank condition $\rank(C)=p$ the equality constrains $Cx=d$ in \eqref{eq:lsle} are consistent, thus LSE \eqref{eq:lsle} has solutions. The second rank condition of \eqref{eq:rank condition} guarantees the uniqueness of the solution to  \eqref{eq:lsle}. On the other hand, the augmented system also defines the unique solution $x$ as follows
\begin{equation}\label{eq:aug}
\bfA {\bf x} :=\begin{bmatrix}\bf0&\bf0&C\cr
\bf0&I_m& A \cr
C^\rt & A^\rt&\bf0\end{bmatrix}\begin{bmatrix} \lambda\cr r\cr x\end{bmatrix}=\begin{bmatrix}
  d\cr b\cr \bf0
\end{bmatrix}:=\bf b,
\end{equation}
where $A^\rt$ is the transpose of $A$, $I_m$ denotes the $m\times m$ identity matrix, $\bf0$ is the zeros matrix with conformal dimension, $\lambda\in \R^{p}$ is a vector of Lagrange multipliers, and $r$ is the residual vector $r=b-Ax$. As stated in \cite{CoxHigham99,Elden80}, when the rank condition \eqref{eq:rank condition} is satisfied, $\bfA$ is nonsingular and its inverse has the following expression
\begin{equation}\label{eq:lse inverse}
  \bfA^{-1}=\begin{bmatrix}(AC_A^\dagger)^\rt AC_A^\dagger&-(AC_A^\dagger)^\rt&(C_A^\dagger)^\rt\cr
-AC_A^\dagger&I_m-(A\bfP)\bfK& \bfK^\rt \cr
C_A^\dagger & \bfK&-\left((A \bfP)^\rt (A \bfP)\right)^\dagger\end{bmatrix}.
\end{equation}

The LSE problem has many applications such as in the analysis of large scale structures \cite{Barlow1988SISC}, and the solution of the inequality constrained least square problem \cite{LawsonHanson95} etc. The algorithms and perturbation analysis of LSE can be found in several papers \cite{Barlow1988SISC,Bjorck96,CoxHigham99,Elden80,Hammarling87,LawsonHanson95,MalyshevSIAM03,Paige90QR,WeiLSE92} and references therein.

In numerical analysis, condition number is an important research topic, which measures the {\em
worst-case} sensitivity of an input data with respect to {\em small}
perturbations on it; see a recent monograph \cite{Cucker2013Book} and references therein. A problem with large condition number is called {\em ill-posed} problem~\cite{Demmel1987NumerMath}. Also Demmel pointed that the distance of a problem to ill-posed sets is the reciprocal of its condition number. To the best of our knowledge a general theory of
condition numbers was first given by Rice in~\cite{Rice}. Let $\phi
: \R^{s} \rightarrow \R^{t}$ be a mapping, where $\R^s$ and $\R^t$
are the usual $s$- and $t$-dimensional Euclidean spaces equipped
with some norms, respectively. If $\phi$ is continuous and
Fr\'{e}chet differentiable in the neighborhood of $a_0 \in \R^s$
then, according to~\cite{Rice}, the {\em relative normwise condition
number} of $a_0$ is given by
\begin{equation}\label{eq:def_CN}
 {\rm cond}^\phi(a_0):=\lim_{\e \rightarrow 0}\sup_{\|\Delta a\| \leq
 \e}\left(\frac{\|\phi(a_0+\Delta
 a)-\phi(a_0)\|}{\|\phi(a_0)\|}\diagup \frac{\|\Delta
 a\|}{\|a_0\|}\right)=\frac{\|\sd{\phi}(a_0)\|\|a_0\|}{\|\phi(a_0)\|},
\end{equation}
where $\sd{\phi}(a_0)$ is the Fr\'{e}chet derivative of $\phi$ at
$a_0$. Condition number can tell us the loss of the precision in finite precision computation of a problem. With the backward error of a problem, we have the following rule of thumb~\cite{Higham2002Book}
$$
\mbox{ forward error} \lesssim \mbox {conditon number} \times \mbox {backward error},
$$
which can bound the relative error of the computed solution.

In scientific computing, we usually encounter sparse or badly-scaled input data. So the normwise condition number defined in (\ref{eq:def_CN}) may overestimate the true conditioning of the problem since it does not take account of the structure of both input and output data with
respect to scaling and/or sparsity. In practice, due to rounding errors
and data storage limitation, it is reasonable to measure the input
errors componentwise instead of normwise. The forward error based on the normwise condition number may be overestimated. Since 1980's, componentwise analysis~\cite{Skeel79,Rohn89}, which often gives sharper error bounds, has been used. In fact, most error bounds in LAPACK
\cite{Anderson1999Lapack} are based on componentwise perturbation analysis. There
are two kinds of condition numbers in componentwise analysis: the
mixed condition numbers and componentwise condition numbers
\cite{Gohberg}. The mixed condition numbers use the
componentwise error analysis for the input data, while the normwise
error analysis for the output data. On the other hand, the
componentwise condition numbers use the componentwise error analysis
for both input and output data.   Consequently, the
perturbation bounds based on the mixed and componentwise condition
numbers are more effective and sharper than those based on the
normwise condition number when the data is sparse or badly scaled. An early survey for mixed and componentwise analysis in numerical
linear algebra can be found in~\cite{HighamSurvey94}. Mixed and componentwise condition numbers had been studied extensively in linear least squares problem \cite{CDW07,DiaoWangWeiQiao2013,CuckerDiao2007Calcolo}, total least squares problems \cite{Zhou,LiJia2011,DiaoWeiXie}, indefinite least squares problem \cite{LiWangYang2014ILS}, generalized spectral
              projections and matrix sign functions \cite{WangetalTaiwan2016}, Tikhonov regularization problems \cite{ChuetalTik2011,DiaoWeiQiao2016}, matrix equations \cite{DiaoXiangWei2007,DiaoShiWei2013,WangYangLi2015} and etc.

In this paper, we study the sensitivity  of a linear function of the LSE solution $x$ to perturbations on the date $A$, $C$, $ b$ and $d$, which is defined as
\begin{align}\label{eq:g dfn}
\Psi: \R^{m\times n} \times \R^{p\times n} \times
\R^m  \times \R^p &\rightarrow \R^k\\
\Psi(A,\,C,\,b,\,d)&:=\call\left(\bfK
b+C_A^\dagger d\right),\nonumber
\end{align}
where  $\call$ is an $k$-by-$n$, $k \le n$, matrix introduced for
the selection of the solution components. For example, when
$\call= I_n$ ($k=n$), all the $n$ components of the solution $x$
are equally selected. When $\call= e_i$ ($k=1$), the $i$th unit
vector in $\R^n$, then only the $i$th component of the
solution is selected. In the remainder of this paper, we always suppose that $\call$ is not numerically perturbed.

This paper is devoted to obtain the explicit expressions for   mixed and componentwise condition numbers of the linear function of the solution when perturbations on data are measured componentwise and the perturbations on the solution are measured either componentwise or normwise by means of the dual techniques \cite{22.0}. In particular, as also mentioned in \cite{22.0}, the dual techniques enable us to derive condition numbers by maximizing a linear function over a space of smaller dimension than the data space. And sharp upper bounds also are obtained, which can be estimated efficiently via the classical Hager-Higham algorithm~\cite{HagerCond84,HighamFortran88,HighamSISC90} during using the generalized QR factorization method \cite{CoxHigham99,Hammarling87,Paige90QR}  for solving LSE by means of utilizing the already computed matrix decompositions to reduce the computational complexity of condition estimations. Numerical examples in Section \ref{sec:nume ex} tell us,  under some situations, the  mixed and componentwise condition numbers of the linear function for LSE can be much smaller than the normwise condition numbers given in \cite{CoxHigham99,LiWang2016LSE}. The first order perturbation bounds based on the normwise condition numbers \cite{CoxHigham99,LiWang2016LSE} are pessimistic, while the first order perturbation bounds given by the proposed   mixed and componentwise condition numbers can give sharp perturbation bounds.

The paper is organized as follows. In Section \ref{sec:cond}, the dual techniques for deriving condition number  \cite{22.0} is reviewed and applied to LSE, then we propose the Hager-Higham algorithm \cite{HagerCond84,HighamFortran88,HighamSISC90} to estimate the upper bounds for the mixed and componentwise condition numbers for the linear function of the solution of LSE by taking account of the already computed matrix decompositions during solving LSE by means of the generalized QR factorization method \cite{CoxHigham99,Hammarling87,Paige90QR}. We do some numerical examples to show the effectiveness of the proposed condition numbers in Section \ref{sec:nume ex}. At end, in Section \ref{sec:con} concluding remarks are drawn.

\section{Mixed and componentwise condition numbers for LSE}\label{sec:cond}

In this section we will derive the explicit condition numbers expressions for a linear function of the solution of LSE by means of the dual techniques under componentwise perturbations, which is introduced in \cite{22.0}. Also sharp upper bounds for the mixed and componentwise condition numbers are obtained. Through using the already computed decomposition of the generalized QR factorization method \cite{CoxHigham99,Hammarling87,Paige90QR} for solving LSE, we can estimated upper bounds efficiently via the Hager-Higham algorithm~\cite{HagerCond84,HighamFortran88,HighamSISC90}.

\subsection{Dual techniques}

For the Euclidean spaces $\cX$  and  $\cY$  equipped scalar products $\langle \cdot ,\cdot \rangle_{\cX}$ and  $\langle\cdot,\cdot \rangle_{\cY}$ respectively, let a linear operator $ L:\, \cX \rightarrow \cY $ be well defined. We denote the corresponding norm norms $\|\cdot\|_{\cX} $ and $\|\cdot\|_{\cY}$ respectively. The well-known adjoint operator and dual norm are defined as follows.

\begin{definition}\label{def:adj dual}
The adjoint operator of  $L$,\, $L^{\ast}:\cY \rightarrow \cX $  is defined by
\[
\langle \bv,L\bu \rangle_{\cY}=\langle L^{\ast}\bv,\bu\rangle_{\cX}
\]
\end{definition}
where $\bu\in \cX $ and $\bv \in\cY $.  The dual norm $ \|\cdot\|_{\cX^{\ast}}$  of  $ \|\cdot\|_{\cX} $  is defined by
\[
\|\bu\|_{\cX^{\ast}}=\max_{\bw\neq0}\frac{ \langle  \bu,\bw \rangle _{\cX}}{\|\bw\|_{\cX}}
\]
and the dual norm $\|.\|_{\cY^{\ast}}$ can be defined similarly.

For the common vector norms with respect to the canonical scalar product in $ \R^{n} $, their dual norms  are  given by:
\[
\|\cdot \|_{{1}^\ast}=\|\cdot\|_{\infty}, \quad  \|\cdot \|_{{\infty}^\ast}=\|\cdot\|_{1}\quad \mbox{and } \quad \|\cdot\|_{{2}^\ast}=\|\cdot\|_{2}.
\]

For the matrix norms in $\R^{m\times n}$   with respect to the scalar product $ \langle A,B\rangle = \trace(A^\top B)$, where $\trace(A)$ is the trace of $A$, we have  $ \|A\|_{F}\ast=\|A\|_{F}$ since  $\trace(A^\top A)=\|A\|_{F}^{2}$, where $\|\cdot\|_F$ is Frobenius norm.

For the linear operator from $\cX$  to $\cY$, let $\|\cdot\|_{\cX,\cY}$  be the operator norm induced by
the norms  $\|\cdot\|_{\cX}$ and  $\|\cdot\|_{\cY}$. Consequently, for linear operators from $\cY$ to $\cX$, the norm
induced from the dual norms $\| \cdot \|_{\cX^\ast}$ and
$\| \cdot \|_{\cY^\ast}$, is denoted by $\| \cdot \|_{\cY^\ast,\cX^\ast}$.

We have the following  result for the adjoint operators and dual norms \cite{22.0}.
\begin{lemma}\label{lemma:adjoint}
$$
\|L\|_{\cX,\cY}=\|L^{\ast}\|_{\cY^\ast,\cX^\ast}.
$$
\end{lemma}

As mentioned in \cite{22.0}, it may be more desirable  to compute $ \|L^{\ast}\|_{\cY^\ast,\cX^\ast} $ instead of $\|L\|_{\cX,\cY}$
when the dimension of the Euclidean space $\cY^{\ast}$  is  lower  than  $\cX$ because it implies a maximization over a space of smaller dimension.

Let $ \cX=\cX_{1}\times \cdots\times \cX_{s} $ be a product space, where each Euclidean space  $\cX_{i}$ is equipped with the scalar product  $\langle\cdot,\cdot \rangle_{\cX_{i}}$ and the corresponding norm $\|\cdot\|_{\cX_{i}}$.  The following scalar product
\[
\langle (\bu_{1},\cdots,\bu_{s}),(\bv_{1},\cdots,\bv_{s})\rangle=\langle \bu_{1},\bv_{1}\rangle_{\cX_{1}}+\cdots+\langle \bu_{s},\bv_{s}\rangle_{\cX_{s}},
\]
and the corresponding product norm
\[
\|(\bu_{1},\cdots,\bu_{s})\|_{v}=v(\|\bu_{1}\|_{\cX_{1}},\cdots,\|\bu_{s}\|_{\cX_{s}}),
\]
are defined in  $\cX$,  where   $v$  is an absolute norm \cite{Higham2002Book}  on  $\R^{s}$, that is  $ v(|\bu|)=v(\bu)$, for any  $\bu\in \R^{s}$. We denote   $v^{\ast}$ is the dual norm of $v$ with respect to the canonical inner-product of $ \R^{s} $  and we are interested in determining the dual $ \|\cdot\|_{v^\ast}$ of the product norm $\|\cdot\|_{v}$  with respect to the scalar product of  $\cX$. The following result can be found in \cite{22.0}.

\begin{lemma}\label{lemmaProductNorm}
The dual of the product norm can be expressed by
\[
\|(\bu_{1},\cdots,\bu_{s})\|_{v^\ast}=v(\|\bu_{1}\|_{\cX_{1^\ast}},\cdots,\|\bu_{s}\|_{\cX_{s^\ast}}).
\]
\end{lemma}

In the next subsection, the explicit expressions for the condition numbers of LSE can de derived via the adjoint operators
and dual norms.  Firstly, the Euclidean space $\cX$ with
norm $\| \cdot \|_\cX$ can be regarded as the space of the input data in LSE. Secondly, $\cY$ with norm $\| \cdot \|_\cY$ can be viewed as the space of the solution in LSE. Then
the function $\Psi$ in (\ref{eq:g dfn}) is an operator from $\cX$
to $\cY$ and the condition number is the measurement of the
sensitivity of $\Psi$ to the perturbation in its input data.

Assuming that $\Psi$ is Fr{\'e}chet differentiable in neighborhood of  $\bu \in \cX$, from \cite{Rice} the absolute condition number of  $\Psi$ at $\bu\in \cX$  is given by
\[
\kappa(\bu)=\|\sd \Psi(\bu)\|_{\cX,\cY}=
\max_{\| \bz \|_\cX = 1} \|\sd \Psi( \bu ) \cdot \bz \|_\cY ,
\]
where $ \|\cdot\|_{\cX,\cY} $  is the operator norm induced by the norms  $ \|\cdot\|_{\cX} $ and  $ \|\cdot\|_{\cY}$ and $\sd \Psi(\bu )$ is the Fr{\'e}chet derivative of $\Psi$ at $\bu$. If the output $ \Psi(\bu) $ is nonzero, the \emph{relative condition number} of  $\bu$ at  $\bu\in \cX$ is defined as
\[
\kappa^{\rel}(\bu)=\kappa(\bu)\frac{\|\bu\|_{\cX}}{\|\Psi(\bu)\|_{\cY}}.
\]
The expression of  $\kappa(\bu)$  is related to the operator norm of the linear operator  $\sd \Psi(\bu)$. In view of
Lemma~\ref{lemma:adjoint}, the following expression of
$\kappa(\bu)$
\begin{equation} \label{eqnK}
 \kappa(\bu)=\max_{\|\sd \bu\|_{\cX}=1} \|\sd \Psi(\bu)\cdot \sd \bu\|_{\cY}=\max_{\|\bz \|_{\cY^\ast}=1}\|\sd \Psi(\bz)^{\ast}\cdot \bz\|_{\cX^\ast},
\end{equation}
can be deduced  in terms of adjoint operator and dual norm.

For a data space $ \cX=\R^{n}$, the following componentwise metric will be defined. For any input data $ \bu\in \cX$,  the subset $ \cX_{\bu} \in \cX$ is given by
$$
 \cX_{\bu}=\left \{\sd \bu\in \cX~|~ \sd \bu_{i}=0 \mbox{ whenever } \bu_{i}=0, 1\leq i\leq n \right \}.
$$
Thus in the situation of a componentwise perturbation analysis, the perturbation  $\sd \bu\in \cX_{\bu}$ of  $\bu$ is measured by using the following componentwise norm
\begin{equation}\label{eq:rr}
 \|\sd \bu\|_{c}=\min\{\omega,|\sd \bu_{i}|\leq \omega |\bu_{i}|,i=1,\ldots, n\},
\end{equation}
with respect to $\bu$. Equivalently, it is not difficult to check that  the componentwise relative norm has the following property
\begin{equation}\label{eqnCNorm}
\|\sd \bu\|_{c}=\max\left\{\frac{|\sd \bu_{i}|}{|\bu_{i}|},\bu_{i}\neq 0\right\}=\left\|\left(\frac{|\sd \bu_{i}|}{|\bu_{i}|}\right)\right\|_{\infty},
\end{equation}
where $\sd\bu \in \cX_\bu$.

In the next step, the explicit expression of the dual norm $\| \cdot \|_{c^\ast}$
of the componentwise norm $\| \cdot \|_c$ is given for the special choice of $\cX$, $\cX_i$, the absolute norm $v$ and the norm on $\cX_i$. We fix that the product space $\cX=\mathbb{R}^n$, $\cX_i=\mathbb{R}$, and the absolute norm $v=\| \cdot \|_{\infty}$. Setting the norm $\| \sd \bu_i \|_{\cX_i}$ in $\cX_i$ to
$| \sd \bu_i | / | \bu_i |$ when $\bu_i \neq 0$, from
Definition~\ref{def:adj dual}, we have the dual norm
\[
\| \sd \bu_i \|_{\cX_i^\ast} =
\max_{\bz \neq 0} \frac{| \sd \bu_i \cdot \bz |}{\| \bz \|_{\cX_i}} =
\max_{\bz \neq 0} \frac{| \sd \bu_i \cdot \bz |}{|\bz| / |\bu_i|} =
|\sd \bu_i| \, |\bu_i| .
\]
From Lemma~\ref{lemmaProductNorm} and (\ref{eqnCNorm}) and the property
$\| \cdot \|_{\infty^*} = \| \cdot \|_1$,  the explicit expression of the dual norm
\begin{equation} \label{eqnDualC}
\| \sd \bu \|_{c^\ast} =
\| (\|\sd \bu_1\|_{\cX^\ast} ,..., \|\sd \bu_n\|_{\cX^\ast}) \|_{\infty^\ast} =
\| (|\sd \bu_1| \, |\bu_1| ,..., |\sd \bu_n| \, |\bu_n|) \|_1
\end{equation}
is derived.

When we consider the componentwise perturbation $\sd \bu$ of the input data $\bu $, we always assume that $\sd \bu \in \cX_\bu$ and use  $\|\sd \bu\|_{c}$ defined in \eqref{eq:rr} to measure it. Following
(\ref{eqnK}), we have the following lemma on the explicit expression of the condition
number in terms of adjoint operator and dual norm.


\begin{lemma} \label{lemmaK}
Using the above notations and the componentwise norm defined
in (\ref{eqnCNorm}), the condition number $\kappa(\bu)$ can be
expressed by
\[
\kappa(\bu) =
\max_{\| \bz \|_{\cY^\ast} = 1}
\| (\sd \Psi(\bu))^* \cdot \bz \|_{c^\ast} ,
\]
where $\| \cdot \|_{c^\ast}$ is given by (\ref{eqnDualC}).
\end{lemma}

In the next subsection, we derive the explicit expressions for condition numbers based on Lemma \ref{lemmaK}. We always measure the errors for the solution under componentwise perturbation analysis, while for the input data, the error can be measured either componentwisely or normwisely.


\subsection{Deriving condition number expressions via dual techniques}

In this subsection we will derive the explicit expressions of condition numbers for LSE through dual techniques stated in the previous subsection. Firstly, we prove the linear function $\Psi$ defined by \eqref{eq:g dfn}  is Fr\'{e}chet
differentiable  and its Fr\'{e}chet derivative is obtained by means of matrix differential calculus \cite{Magnus} in Lemma \ref{lemma:gLSQI}. Before that, we need the following lemma.

\begin{lemma}{\rm\cite[Page 171, Theorem 3]{Magnus}}\label{Lemma:inverse df}
  Let $T$ be the set of non-singular real $m\times m$ matrices, and $S$ be an open subset of $\R^{n\times q}$. If the matrix function
  $F:S\rightarrow T$ is $k$ times (continuously) differentiable on $S$, then so is the matrix function $F^{-1}:S\rightarrow T$ defined by $F^{-1}(X)=\left(F(X)\right)^{-1}$, and
  $$
  \sd F^{-1}=-F^{-1}(\sd F)F^{-1}.
  $$
\end{lemma}

\begin{lemma}\label{lemma:gLSQI}
The function $\Psi$ is a continuous mapping on $\R^{m\times n} \times \R^{p\times n} \times
\R^m  \times \R^p$. In addition, $\Psi$ is Fr\'{e}chet
differentiable at $(A,\, C,\, b,\, d)$ and its Fr\'{e}chet derivative is given by
\begin{align}\label{eqnJ}
J&:=\sd {\Psi}(A,\,C,\,b,\,d)\cdot (\sd A,\, \sd C,\, \sd b,\, \sd d)=-\call\bfK\sd A x+ \call \bfK\bfK^\rt
(\sd A)^\rt r\nonumber \\
&\quad - \call C_A^\dagger
\sd C x-\call \bfK\bfK^\rt(\sd C)^\rt (AC_A^\dagger)^\top r
+\call \bfK\sd b+\call C_A^\dagger \sd d \nonumber \\
&:=J_{1}(\sd A)+J_{2}(\sd C)+J_{3}(\sd b)+J_{4}(\sd d),
\end{align}
where $\sd A \in \R^{m\times n},\, \sd C \in ^{p\times n}, \, \sd b\in \R^m$ and $\sd d \in \R^p$.
\end{lemma}

\begin{proof} From \eqref{eq:aug}, we know that ${\bf x}=\bfA^{-1} \bf b$. Since $\cal A$ is invertible, the linear operator $\Psi$ defined in \eqref{eq:g dfn} is continuously Fr{\'e}chet differentiable  in a neighborhood of the data $(A,\,C,\, b,\,d)$ from the theory of the matrix differential calculus \cite{Magnus}. With Lemma~\ref{Lemma:inverse df}, we can deduce that
$$
\sd {\bf x}=(\sd \bfA^{-1}) \bf b +\bfA^{-1} \sd \bf b =-\bfA^{-1} (\sd \bfA) \bf x+\bfA^{-1} \sd \bf b.
$$
Noting that
$$
\sd {\bf x}=\begin{bmatrix} \sd \lambda\cr \sd r\cr \sd x \end{bmatrix},\quad {\sd \bf b}=\begin{bmatrix}
  \sd d\cr \sd b\cr \bf0
\end{bmatrix},\quad \sd \bfA=\begin{bmatrix}\bf0&\bf0&\sd C\cr
\bf0&\bf0& \sd A \cr
\sd C^\rt & \sd A^\rt&\bf0\end{bmatrix},
$$
recalling \eqref{eq:lse inverse}, and after some algebraic operations we deduce that
\begin{align}\label{eq:lse dx}
\sd x & =\bfK\sd b+C_A^\dagger \sd d- C_A^\dagger
\sd C x -\bfK\sd A x+ ((A\bfP)^\rt (A\bfP))^\dagger
(\sd A)^\rt r\\
&+((A\bfP)^\rt (A\bfP))^\dagger (\sd C)^\rt \lambda.\nonumber
\end{align}
Thus $\sd \Psi=\call \sd x$ since $\Psi$ is linear. From \eqref{eq:aug}, it can verified that $C^\rt \lambda +A^\rt r=0$, since $C$ has full row rank ($CC^\dagger=I_{p}$) it is easy to see that $
\lambda=-(AC^\dagger)^\rt r$. Substituting the above expression in \eqref{eq:lse dx}, using the equality $(AC^\dagger)^\top r=(AC_A^\dagger)^\top r$ from Lemma 4.2 in \cite{CoxHigham99}, and noting $((A\bfP)^\rt (A\bfP))^\dagger=\bfK\bfK^\rt$, we can complete the proof of this lemma.
\hfill $\Box$
\end{proof}

Using the definition of the adjoint operator and
the classical definition of the scalar
product in the data space $\R^{m\times n} \times \R^{p\times n} \times
\R^m  \times \R^p$, an explicit
expression of the adjoint operator of the above $J(\sd A,\, \sd C, \, \sd b, \sd d)$ is given in the following lemma.

\begin{lemma}\label{lemmaDualJ}
The adjoint of operator of the Fr{\'e}chet derivative $J(\sd A,\, \sd C, \, \sd b, \sd d)$ in \eqref{eqnJ} is given by
\begin{align*}
J^{\ast}&:\R^{k}\rightarrow \R^{m\times n} \times \R^{p\times n} \times
\R^m  \times \R^p\\
 &u\mapsto \left(r u^\top\call \bfK\bfK^\rt - \bfK^\rt\call^\rt u x^\top,\,-xu^\top\call C_A^\dagger -   \bfK\bfK^\rt \call^\top u r^\top AC_A^\dagger,\,\bfK^\top \call^\top  u,\right.\\
 &\left.\quad \quad \quad (C_A^\dagger)^\top \call^\top  u\right).
\end{align*}
\end{lemma}

\begin{proof}
Using \eqref{eqnJ} and the definition of the scalar product in the matrix space, for any $u\in \R^{k}$, we have
\begin{align*}
\langle u,J_{1}(\sd A)\rangle=&u^\top(\call \bfK\bfK^\rt
(\sd A)^\rt r-\call\bfK\sd A x)\\
=&\trace(ru^\top\call \bfK\bfK^\rt
(\sd A)^\rt )-\trace(xu^\top \call\bfK\sd A)\\
=&\langle r u^\top\call \bfK\bfK^\rt - \bfK^\rt\call^\rt u x^\top, \sd A\rangle.
\end{align*}
For the second part of the adjoint of the derivative  $J$, we have
\begin{align*}
\langle u,J_{2}(\sd C)\rangle=&-u^\top( \call C_A^\dagger
\sd C x+\call \bfK\bfK^\rt(\sd C)^\rt (AC_A^\dagger)^\top r)\\
=&-\trace(xu^\top\call C_A^\dagger
\sd C )-\trace( (AC_A^\dagger)^\top r u^\top \call \bfK\bfK^\rt(\sd C)^\rt)\\
=&-\langle xu^\top\call C_A^\dagger +   \bfK\bfK^\rt \call^\top u r^\top AC_A^\dagger, \sd C\rangle.
\end{align*}
For the third part of the adjoint of the derivative  $J$, we have
\begin{align*}
\langle u,J_{3}(\sd b)\rangle
=&u^\top \call \bfK\sd b=\langle  \bfK^\top \call^\top  u, \sd b\rangle.
\end{align*}
Similarly,  for the fourth part of the adjoint of the derivative  $J$, we have
\begin{align*}
\langle u,J_{4}(\sd d)\rangle
=&u^\top \call C_A^\dagger \sd d=\langle  (C_A^\dagger)^\top \call^\top  u, \sd d\rangle.
\end{align*}
Let
\begin{align*}
J_1^*({u})& =r u^\top\call \bfK\bfK^\rt - \bfK^\rt\call^\rt u x^\top,\quad J_2^*({u}) =-\left(xu^\top\call C_A^\dagger +   \bfK\bfK^\rt \call^\top u r^\top AC_A^\dagger\right),\\
J_3^*({u})& =\bfK^\top \call^\top  u,\quad J_4^*({u}) =(C_A^\dagger)^\top \call^\top  u,
\end{align*}
then
\begin{align*}
\langle J^*({u}),\ (\sd A,\, \sd C,\, \sd b,\, \sd d) \rangle &=
\langle (J_1^*({u}),\, J_2^*({u}),\, J_3^*({u}), \, J_2^*({u})) , \
(\sd A,\, \sd C,\, \sd b,\, \sd d) \rangle\\
& = \langle {u},\ J(\sd A,\, \sd C,\, \sd b,\, \sd d)  \rangle,
\end{align*}
which completes the proof. \hfill $\Box$
\end{proof}

After obtaining an explicit expression of the adjoint operator
of the Fr{\'e}chet derivative,  we now give an explicit expression
of the condition number $\kappa$ (\ref{eqnK}) in terms
the dual norm in the solution space in the following theorem. In the following, if $A\in \R^{m\times n}$ and $B \in \R^{p \times q}$, then the
{\em Kronecker product} $A \otimes B\in\R^{mp\times nq}$ is
defined by $A \otimes B = \left[a_{ij}B\right] \in \R^{mp\times nq}$~\cite{Graham1981book}. We denote $\vect(A)$ as the vector obtained by stacking
the columns of a matrix $A$  \cite{Graham1981book} and $D_A$ denotes the diagonal matrix
$\mathrm{diag}(\vect(A))$.

\begin{theorem} \label{thmK}
The condition number for the LSE
problem can be expressed by
\[
\kappa = \max_{\| {u} \|_{\cY^\ast} = 1}
\left\| [ \bfH D_A \ \ \bfJ D_{{C}}\ \ \bfK D_{{b}}\ \ C_A^\dagger D_{{d}}]^{\top}
\call^\top\right\|_{\cY^\ast,1} ,
\]
where
\begin{align} \label{eqnV}
 \bfH= (\bfK\bfK^\rt)\otimes
r^\rt-x^\rt\otimes \bfK,\quad \bfJ= x^\rt\otimes C_A^\dagger + (\bfK\bfK^\rt)\otimes
(r^\rt AC_A^\dagger).
\end{align}
\end{theorem}

\begin{proof}
Let $\sd a_{ij}$, $\sd c_{ij}$,  $\sd b_{ij}$ and  $\sd d_{i}$ be the entries of $\sd A, \,\sd C, \,\sd b $ and  $\sd d$ respectively, using \eqref{eqnDualC}, we have
 \[
 \|(\sd A,\sd C,\sd b,\sd d)\|_{c^\ast}=\sum_{i,j}|\sd a_{ij}||a_{ij}|+\sum_{i,j}|\sd c_{ij}||c_{ij}|+\sum_{i}|\sd b_{i}||b_{i}|+\sum_{i}|\sd d_{i}||d_{i}|.
 \]
Applying Lemma~\ref{lemmaDualJ}, we derive that
\begin{align*}
\|J^{\ast}(u)\|_{c\ast}=&\sum_{j=1}^{n}\sum_{i=1}^{m}  |a_{ij}|\left|\left(r u^\top\call \bfK\bfK^\rt - \bfK^\rt\call^\rt u x^\top\right)_{ij}\right|+\sum_{i=1}^{m}
|b_{i}|\left|\left(\bfK^\top \call^\top  u\right)_{i}\right|
\\+&\sum_{j=1}^{n}\sum_{i=1}^{p}  |c_{ij}|\left|\left(xu^\top\call C_A^\dagger + \bfK\bfK^\rt \call^\top u r^\top AC_A^\dagger\right)_{ij}\right|+\sum_{i=1}^{p}
|d_{i}|\left|\left((C_A^\dagger)^\top \call^\top  u\right)_{i}\right|\\
=&\sum_{j=1}^{n}\sum_{i=1}^{m}  |a_{ij}|\left|\left(r_i  (\bfK\bfK^\rt e_j)^\top - x_j(\bfK e_i)^\rt \right)\call^\top u\right|+\sum_{i=1}^{m}
|b_{i}|\left|\left(\bfK e_i\right)^\top \call^\top u\right|
\\+&\sum_{j=1}^{n}\sum_{i=1}^{p}  |c_{ij}|\left|\left(x_i  (C_A^\dagger e_j)^\top + (r^\top AC_A^\dagger)_j(\bfK\bfK e_i)^\rt  \right)\call^\top u\right|\\
+&\sum_{i=1}^{p}
|d_{i}|\left|(C_A^\dagger e_i)^\top  \call^\top u\right|
\end{align*}
where $r_{i}$ is the $i$th component of $r$. Then it can be verified that $r_i  (\bfK\bfK^\rt e_j)^\top- x_j(\bfK e_i)^\rt $
is the $(m\,(j-1)+i)$th column of the $n \times (mn)$ matrix
$\bfH$
and $x_i  (C_A^\dagger e_j)^\top + (r^\top AC_A^\dagger)_j(\bfK\bfK e_i)^\rt $ is the $(p\,(i-1)+j)$th
column of the $n \times (np)$ matrix
$\bfJ$ in (\ref{eqnV}), implying that the above expression
equals
\[
\left\| \left[ \begin{array}{c}
D_A \bfH^\top \call {u} \\
D_{{C}} \bfJ^\top \call {u}\\
D_{{b}} \bfK^\top \call {u}\\
D_{{d}}( C_A^\dagger)^\top \call {u}
\end{array} \right] \right\|_1 =
\left\| [ \bfH D_A \ \ \bfJ D_{{C}}\ \ \bfK D_{{b}}\ \ C_A^\dagger D_{{d}}]^{\top}
\call^\top {u} \right\|_1 .
\]
The theorem then follows from Lemma~\ref{lemmaK}. \hfill $\Box$
\end{proof}

The following case study discusses some commonly used norms
for the norm in the solution space to obtain some specific
expressions of the condition number $\kappa$. It is not difficult to prove the following corollary and its proof is omitted.

\begin{corollary} \label{colKinfty1}
Using the above notations,
when the infinity norm is chosen as the norm in the solution
space $\cY$, we get
\begin{equation} \label{eqnKinfty1}
\kappa_{\infty} =\left\| \left|\call \bfH\right| \vect (|A|)
+\left|\call \bfJ\right| \vect (|C|)+ |\call \bfK| |b| +
|\call  C_A^\dagger| |d|\right\|_\infty,
\end{equation}
where $|\cB|=(|b_{ij}|)$, $b_{ij}$ is the $(i,j)$th entry of $\cB$.
\end{corollary}

When the infinity norm is chosen as the norm in the solution
space $\R^n$, the corresponding relative mixed condition number is given by
\begin{equation}
\kappa_{\infty}^{\rel} =\frac{\left\| \left|\call \bfH\right| \vect (|A|)
+\left|\call \bfJ\right| \vect (|C|)+ |\call \bfK| |b| +
|\call  C_A^\dagger| |d|\right\|_\infty} {\|\call x\|_\infty},
\end{equation}

In the following, we consider the 2-norm on the solution space and derive an upper bound for the corresponding condition number respect to the 2-norm on the solution space.

\begin{corollary} \label{colK2}
When the 2-norm is used in the solution space, we have
\begin{equation} \label{eqnK2}
\kappa_2 \le
\sqrt{k} \, \kappa_{\infty} .
\end{equation}
\end{corollary}

\begin{proof}
When $\| \cdot \|_\cY= \| \cdot \|_2$, then
$\| \cdot \|_{\cY^\ast} = \| \cdot \|_2$. From Theorem~\ref{thmK},
\[
\kappa_2 = \left\| [ \bfH D_A \ \ \bfJ D_{{C}}\ \ \bfK D_{{b}}\ \ C_A^\dagger D_{{d}}]^{\top}
\call^\top\right\|_{2,1} .
\]
It follows from \cite{Higham2002Book} that for any matrix $B$,
$\| B \|_{2,1} = \max_{\| {u} \|_2 = 1} \| B {u} \|_1
= \| B \hat{{u}} \|_1$, where $\hat{{u}} \in \mathbb{R}^k$
is a unit 2-norm vector. Applying $\| \hat{{u}} \|_1 \le
\sqrt{k}\, \| \hat{{u}} \|_2$, we get
\[
\| B \|_{2,1} = \| B \hat{{u}} \|_1 \le
\| B \|_1 \| \hat{{u}} \|_1 \le
\sqrt{k}\, \| B \|_1 .
\]
Substituting the above $B$ with $[ VD_A\ \
W D_{b}]^{\top} L$, we have
\[
\kappa_2 \le \sqrt{k} \,
\left\| [ \bfH D_A \ \ \bfJ D_{{C}}\ \ \bfK D_{{b}}\ \ C_A^\dagger D_{{d}}]^{\top}
\call^\top\right\|_1 ,
\]
which implies (\ref{eqnK2}). \hfill $\Box$
\end{proof}

By now, we have considered the various mixed condition
numbers, that is, componentwise norm in the data space and
the infinity norm or 2-norm in the solution space. In the
rest of the subsection, we study the case of componentwise
condition number,
that is, componentwise norm in the solution space as well.

\begin{corollary} \label{colKc}
Considering the componentwise norm defined by
\begin{equation}\label{eq:comp norm}
\| {u} \|_c =
\min \{ \omega , \ | u_i | \le \omega \,
|(\call {x})_i|, \ i=1,...,k \} =
\max \{ |u_i| / |(\call {x})_i|, \ i=1,...,k \} ,
\end{equation}
in the solution space, we have the following three expressions
for the componentwise condition number
\begin{eqnarray*}
\kappa_c
&=& \| D_{\call {x}}^{-1}
\call [ \bfH D_A \ \ \bfJ D_{{C}}\ \ \bfK D_{{b}}\ \ C_A^\dagger D_{{d}}]
 \|_{\infty} \\
&=& \| |D_{\call {x}}^{-1}( \left|\call \bfH\right| \vect (|A|)
+\left|\call \bfJ\right| \vect (|C|)+ |\call \bfK| |b| +
|\call  C_A^\dagger| |d|) \|_{\infty}.
\end{eqnarray*}
\end{corollary}

\begin{proof}
The expressions immediately follow from Theorem~\ref{thmK}
and Corollary \ref{colKinfty1}. \hfill $\Box$
\end{proof}

\subsection{Condition estimations}
In this subsection, we will derive the sharp upper bounds for $\kappa_\infty^{\rel}$ and $\kappa_c$, which can be estimated efficiently by the Hager-Higham algorithm~\cite{HagerCond84,HighamFortran88,HighamSISC90} during using the generalized QR factorization method  (GQR)  \cite{CoxHigham99,Hammarling87,Paige90QR,AndersonBai92LAA} to solve LSE.

We first review the generalized QR factorization method for solving LSE. Let $A\in \R^{m\times n}$ and $C\in \R^{p\times n}$ with $m+p\geq n\geq p$. the generalized QR factorization  was introduced by Hammarling \cite{Hammarling87} and Paige \cite{Paige90QR}, which further was analyzed by Anderson et al. \cite{AndersonBai92LAA}. There are orthogonal matrices $Q\in \R^{n\times n}$ and $U\in \R^{m\times m}$ such that
\begin{equation}\label{eq:GQR}
U^\top AQ=\bordermatrix{&p&n-p  \cr
                m-n+p&L_{11}& {\bf 0}\cr
                n-p&L_{21}&L_{22}\cr
                }, \quad CQ=\bordermatrix{&p&n-p  \cr
                &S&{\bf 0}
                }
\end{equation}
where $L_{22}\in\R^{(n-p)\times(n-p)}$ and $S\in \R^{p\times p}$ are lower triangular. If rank condition \eqref{eq:rank condition} holds, then $L_{22},\, S$ are nonsingular \cite[Theorem 2.1]{CoxHigham99}. The generalized QR factorization method for solving LSE can be summarized as follows. Let $y_1\in \R^p$ be the solution of the triangular system $Sy_1=d$ and $y_2$ be the solution to the triangular system $L_{22}y_2 =c_2 -L_{21}y_1,$ where
$$
c=U^\rt b=\bordermatrix{ & \cr
                m-n+p& c_1\cr
                n-p&c_2\cr
                }.
$$
Then the solution $x$ to LSE can be computed by $x=Qy$, where $y=[y_1^\top \, y_2^\top]^\top$. The flops of the generalized QR factorization method are $2mn^2+4mnp+2np^2-2mp^2-2n^3/3-2p^3/3$ \cite[Table 2.1]{CoxHigham99}. Thus during GQR method, the decomposition \eqref{eq:GQR} has already been computed, which can be utilized to devise the method based on the Hager-Higham algorithm to estimate  the upper bounds for $\kappa_\infty^{\rel}$ and $\kappa_c$.

In the following, we will give upper bounds for $\kappa_\infty^{\rel}$ and $\kappa_c$, which can be estimated efficiently by the Hager-Higham algorithm~\cite{HagerCond84,HighamFortran88,HighamSISC90}. Firstly, note that for any
matrix $B\in\R^{p\times q}$ and diagonal matrix
$D_v\in\R^{q\times q}$,
\begin{equation*}\label{eq:norm diagonal}
  \|BD_v\|_\infty
  =\|\,|BD_v|\,\|_\infty
  =\|\,|B|\,|D_v|\,\|_\infty
  =\|\,|B||D_v|\bfe\,\|_\infty
  =\left\|\,\left|B\right||v|\right  \|_\infty.
\end{equation*}
where ${\bf e}=[1,\,\ldots,1]^\rt \in \R^q$. With the above property and triangle inequality, we can prove the following theorem and its proof is omitted.

\begin{corollary}\label{co:lsle}
With the notations above, denoting
\begin{align*}
\kappa_\infty^\u&=\frac{\left\|\call \bfK D_{|A||x|}\right\|_\infty+\left\|\call \bfK \bfK^\top D_{|A^\top||r|}\right\|_\infty
+\left\|\call C_A^\dagger D_{|C||x|}\right\|_\infty} {\|\call x\|_\infty}\\
&\quad + \frac{\left\|\call \bfK \bfK^\top D_{|C^\top||(AC_A^\dagger)^\top r|}\right\|_\infty+\left\|\call \bfK D_b \right\|_\infty+\left\|\call C_A^\dagger D_d\right\|_\infty}{\|\call x\|_\infty},\cr
\kappa_c^\u&=\left\|D_{\call x}^{-1}\call \bfK D_{|A||x|}\right\|_\infty+\left\|D_{\call x}^{-1}\call \bfK \bfK^\top D_{|A^\top||r|}\right\|_\infty
+\left\|D_{\call x}^{-1}\call C_A^\dagger D_{|C||x|}\right\|_\infty\\
&\quad +\left\|D_{\call x}^{-1} \call \bfK \bfK^\top D_{|C^\top||(AC_A^\dagger)^\top r|}\right\|_\infty  + \left\|D_{\call x}^{-1}\call \bfK D_b \right\|_\infty+\left\|D_{\call x}^{-1}\call C_A^\dagger D_d\right\|_\infty,
\end{align*}
we have
 \begin{align*}
   \kappa_\infty^{\rel}\leq \kappa_\infty^\u,\quad  \kappa_c\leq \kappa_c^\u.
  \end{align*}
\end{corollary}

\begin{remark}
  From the constructed example in Section \ref{sec:nume ex}, the upper bounds $\kappa_\infty^\u$ and $\kappa_c^\u $ are attainable, thus they are sharp.
\end{remark}

If the factorization \eqref{eq:GQR} is computed, the following expressions can be verified:
\begin{align}\label{fac:KCA}
  \bfK&=Q\begin{bmatrix}0&0\cr 0& L_{22}^{-1}\end{bmatrix}U^\top, \quad \bfK \bfK^\top =Q\begin{bmatrix}0&0\cr 0& L_{22}^{-1} L_{22}^{-\top}\end{bmatrix}Q^\top,  \cr
  C_A^\dagger&=Q\begin{bmatrix}I_p \cr -L_{22}^{-1}L_{21}\end{bmatrix}S^{-1}, \quad (AC_A^\dagger)^\top r=S^{-\top}L_{11}^\top(c_1-L_{11}y_1).
\end{align}

For the each terms in $\kappa_\infty^\u$ and $\kappa_c^\u $, we can use  the classical condition estimation
method~\cite{HagerCond84,HighamFortran88,HighamSISC90} to estimate
them. This method is an efficient method for estimate one-norm of a matrix $\cB$, which involves a sequences of matrix multiplications $\cB v$ and $\cB^\top v$. By taking account of the decompositions \eqref{fac:KCA}, these matrix multiplications can be computed through solving some triangular linear system with different right hands. Thus the computational complexity of the algorithms to estimate $\kappa_\infty^\u$ and $\kappa_c^\u $ can be reduced significatively compared the GQR method.  The detailed descriptions of the Hager-Higham algorithm~\cite{HagerCond84,HighamFortran88,HighamSISC90} to estimate $\kappa_\infty^\u$ and $\kappa_c^\u $  are omitted.

\section{Numerical examples}\label{sec:nume ex}

In this section we test some numerical examples to validate the previous derived results. All the
computations are carried out using \textsc{Matlab} 8.1 with the machine precision
$\mu=2.2 \times 10^{-16}$.

Let $v$ be a $4\times 1$ vector with  $v_4=1/\eta$ where $\eta$ is a small positive number, and other components are set to 1. We construct the data $A,\, C, \,b$ and $d$ as folows
$$
A=\begin{bmatrix}
  1&0&0&0\cr
  0&0&0&0\cr
   0&1&0&0\cr
    0&0&0&0\cr
     0&0&0&0\cr
      0&0&0&0\cr
       0&0&\delta &0\cr
        0&0&0&0\cr
         0&0&0&\delta
\end{bmatrix}\in \R^{9\times 4},\, C=\begin{bmatrix}
   0&1&0&0\cr
    1&0&0&0\cr
\end{bmatrix}\in \R^{2\times 4},\, b=A\cdot v+10^{-5}\cdot b_2,\, d=\begin{bmatrix}
  1\cr 1
\end{bmatrix},
$$
where $b_2$ is an unitary vector satisfying $A^\top b_2=\bf 0$ and $\delta$ is a small positive number to control the conditioning of the augmented matrix $\cA$ defined in \eqref{eq:aug}. Obviously, the matrix $A$ has many zero components and is bad scaled because of the appearance of $\delta$. Thus, it is reasonable to measure the error on the input date by using componentwise perturbation analysis instead  of the normwise perturbation analysis. Also it can be verified that the rank conditions \eqref{eq:rank condition} are satisfied. For the perturbations, we generate them as
\begin{align}\label{eq:pert}
\Delta A&=10^{-8} \cdot \Delta A_{1}\odot A,\, \Delta C=10^{-8} \cdot \Delta C_{1}\odot C,\\
\Delta b&=10^{-8}\cdot \Delta b_{1}\odot  b,\, \Delta d=10^{-8} \cdot \Delta d_{1}\odot  d,\nonumber
\end{align}
where  each components of $\Delta A_1 \in \R^{9\times 4}$, $\Delta C_1 \in \R^{2\times 4}$, $\Delta b_1 \in \R^{9}$  and $\Delta d_1\in \R^{2}$ are uniformly distributed in the interval $ (-1,1) $,  and $\odot$ denotes the componentwise multiplication of two conformal dimensional matrices. When the perturbations are small enough, we denote the unique solution by $\tilde x$  of the following perturbed LSE problem:
\begin{equation*}
 \min_{\tilde x\in \R^n}\|(A+\Delta A) \tilde x-(b+\Delta b)\|_2 \mbox{ subject to } (C +\Delta C)\tilde x=d+\Delta d.
\end{equation*}
We use the GQR method \cite{CoxHigham99} to compute the solution $x$ and the perturbed solution $\tilde x$ separately. Usually the solution $x$ have badly scaled components,  for example,  the last component of $x$ is order of $1/\eta$ while other components of $x$ are equal to 1.

For the $\call$ matrix in our condition numbers, we choose
\[
\call_0=I_4,\quad
\call_1=\left[\begin{matrix} 1 & 0 & 0 &0 \\
                          0 & 1 & 0 &0\\
                          0 & 0 & 0 &1
                          \end{matrix}\right]\in \R^{3\times 4},\quad
\call_2=\left[\begin{matrix} 0 & 0 &0& 1\end{matrix}\right]\in \R^{1\times 4}.
\]
Thus, corresponding to the above three matrices, the whole
${x}$, the subvector $[x_1, \, x_2 , \, x_3]^{\top }$, and
the component $x_n$ are selected respectively.

We measure the normwise, mixed and componentwise relative errors in
$\call {x}$ defined by
\[
r_2^{\rm rel} =
\frac{\|\call \tilde{{x}} -
\call {x} \|_2}
{\|\call {x} \|_2},\quad
r_\infty^{\rm rel} =
\frac{\|\call \tilde{{x}} -
\call {x} \|_\infty}
{\|\call {x} \|_\infty},\quad
r_c^{\rm rel} =
\frac{\|\call \tilde{{x}} -
\call {x} \|_c}
{\|\call {x} \|_c},
\]
where $\|\cdot\|_c$ is the componentwise norm defined
in \eqref{eq:comp norm}. And
we measure the normwise and componentwise perturbation magnitudes as follows
\begin{eqnarray*}
\epsilon_1 &:=& \min\{\epsilon: \|\Delta A\|_F \leqslant
\epsilon \|A\|_F,
 \|\Delta b\|_2 \leqslant \epsilon \|b\|_2,\, \|\Delta C\|_F \leqslant \epsilon
 \|C\|_F,  \|\Delta d\|_2 \leqslant \epsilon
 \|d\|_2  \},\\
 \epsilon_0 &:=& \min\{\epsilon: |\Delta A| \leqslant
\epsilon |A|,
 |\Delta b| \leqslant \epsilon |b|, |\Delta C| \leqslant \epsilon
 |C|,  |\Delta d| \leqslant \epsilon
 |d|  \},
 \end{eqnarray*}
 where $|\Delta A| \leqslant
\epsilon |A|$ should be understood componentwisely.

 Cox and Higham \cite{CoxHigham99} defined the normwise condition number for LSE as follows
 \begin{equation*}
 {\rm cond}(A,C,b,d):=\lim_{\epsilon_1 \rightarrow 0}\sup\frac{\|\tilde x-x\|_2}{\epsilon_1\, \|x\|_2},
\end{equation*}
and proved that
\begin{align*}
\frac{\|\tilde x-x\|_2}{\|x\|_2}&\leq \kappa_1 \epsilon_1+\Oh(\epsilon_1^2),
\end{align*}
where \begin{align*}
 \kappa_1&=\left(\|C_A^\dagger\|_2\|d\|_2+\|\bfK\|_2\|b\|_2+\left\|x^\rt\otimes C_A^\dagger + \left[(r^\rt AC_A^\dagger)\otimes \left(\bfK\bfK^\rt\right)
\right]\Pi\right\|_2\right\|C\|_F.\cr
&\quad \left.\left\|-x^\rt\otimes \bfK +\left[r^\rt\otimes (\bfK\bfK^\rt)
\right]\Pi\right\|_2\|A\|_F\right)/\|x\|_2
\end{align*}
with $\Pi \in \R^{mn\times mn}$ being the {\em vec-permutation matrix}~\cite{Graham1981book}, and
$$
{\rm cond}(A,C,b,d) \leq \kappa_1 \leq 4 \, {\rm cond}(A,C,b,d).
$$

Li and Wang \cite{LiWang2016LSE} defined the absolute normwise condition number for a linear function of the solution $x$ to  LSE as follows
\begin{equation*}
 \kappa_2^{\rm abs}:=\lim_{\epsilon_2 \rightarrow 0}\sup\frac{\|\call (\tilde x-x)\|_2}{\epsilon_2},
\end{equation*}
where
$$
\epsilon_2:=\sqrt{\alpha_A^2\|\Delta A\|_F^2+\alpha_C^2\|\Delta C\|_F^2+\alpha_b^2\|\Delta b\|_2^2+\alpha_d^2\|\Delta d\|_2^2}
$$
with $\alpha_A>0,\, \alpha_C>0,\,\alpha_b>0$ and $\alpha_d>0$. So the explicit expression for the relative normwise condition number  for a linear function of the solution $x$ to  LSE was given by
\begin{align*}
\kappa_2&:=\lim_{\epsilon_2 \rightarrow 0}\sup\frac{\|\call (\tilde x-x)\|_2}{\epsilon_2\, \|\call x\|_2}\cdot \sqrt{\alpha_A^2\|A\|_F^2+\alpha_C^2\|C\|_F^2+\alpha_b^2\|b\|_2^2+\alpha_d^2\|d\|_2^2}\\
&=\frac{\|{K}\|_2^{1/2}}{\|\call x\|_2}\cdot\sqrt{\alpha_A^2\|A\|_F^2+\alpha_C^2\|C\|_F^2+\alpha_b^2\|b\|_2^2+\alpha_d^2\|d\|_2^2},
\end{align*}
where
\begin{align*}
  K&=\left(\frac {\|r\|_2^2} {\alpha_A^2} +\frac {\|r^\top AC_A^\dagger\|_2^2} {\alpha_C^2}\right) \call \left(\left((A\cP)^\top A\cP \right)^\dagger\right)^2 \call^\top\\
  &\quad+\left(\frac {\|x\|_2^2} {\alpha_A^2} +\frac {1} {\alpha_b^2}\right) \call \left(\left((A\cP)^\top A\cP \right)^\dagger\right) \call^\top\\
  &\quad+\left(\frac {\|x\|_2^2} {\alpha_C^2} +\frac {1} {\alpha_d^2}\right) \call C_A^\dagger (C_A^\dagger)^\top \call^\top+\frac{1}{\alpha_C^2}\call \left((A\cP)^\top A\cP \right)^\dagger xr^\top AC_A^\dagger (C_A^\dagger)^\top \call^\top\\
  &\quad+\frac{1}{\alpha_C^2}\call C_A^\dagger (C_A^\dagger)^\top A^\top rx^\top \left((A\cP)^\top A\cP \right)^\dagger \call^\top,
\end{align*}
and in the rest of this section we always set $\alpha_A=\alpha_C=\alpha_b =\alpha_d=1$.  From the definition of condition numbers, the following quantities
 $$
 \epsilon_1 \kappa_1,\,\epsilon_2 \kappa_2,\,
 \epsilon_0 \kappa_\infty^\rel,\,
 \epsilon_0 \kappa_c,
 $$
are the linear normwise, mixed and componentwise  asymptotic perturbation bounds, respectively.

In Table \ref{tab:T1}, we do numerical experiments for different choices of $\eta$ and $\delta$. It can be observed that when $\delta$ deceases from $10^{-3}$ to $10^{-6}$, ${\rm cond}(\cA)$ increases from $\Oh(10^6)$ to $\Oh(10^{12})$ which means that LSE tends to be more ill-conditioned with respect to the decreasing of $\delta$, while the changes of $\eta$ have little effects on the conditioning of $\cA$ because only the right hand side $b$ verifies when $\eta$ changes. In general, the perturbation magnitudes $\epsilon_0,\, \epsilon_1$ and $\epsilon_2$ are $\Oh(10^{-8})$ since the perturbations are generated via \eqref{eq:pert}. Thus, the linear normwise, mixed and componentwise  asymptotic perturbation bounds can bound the exact relative perturbation bounds for the solution, respectively. However, the linear normwise asymptotic perturbation bounds given by $\epsilon_1 \kappa_1$ and $\epsilon_2 \kappa_2$ can severely overestimate the relative normwise perturbation bounds $r_2^{\rm rel}$. For example, when $\eta=10^{-6}$, $\delta =10^{-6}$ and $\call=\call_1$, $\epsilon_2 \kappa_2\approx 1.7321\cdot 10^4$ while $r_2^{\rm rel}=4.9302\cdot 10^{-9}$. Also for $\kappa_1$, when $\eta=10^{-6}$, $\delta =10^{-6}$ and $\call=I_4$, $\epsilon_1 \kappa_1\approx 1.4142\cdot 10^{-2}$ while $r_2^{\rm rel}=4.1157\cdot 10^{-9}$. Thus the normwise condition number $\kappa_2$ for a linear function of $x$ proposed by Li and Wang \cite{LiWang2016LSE} is not effective, and cannot be used to measure the conditioning of the solution to LSE in this situation. On the other hand, from the values of $r_2^\rel$, $r_m^\rel$ and $r_c^\rel$, one can see that different components of $x$ have different relative errors, which means that we should measure the conditioning of components of $x$ through defining a linear function for $x$ given by \eqref{eq:g dfn} to select the interested component. This validates the motivations of this paper. For all $\eta$, $\delta$ and $\call$, the mixed and componentwise condition numbers are equal to $2$, which tell us that these special constructed LSE problems are well-conditioned under componentwise perturbation analysis. Moreover, the linear mixed and componentwise asymptotic perturbation bounds coincide with the relative mixed and componentwise errors for all choices of $\call$. The derived upper bounds  $\kappa_\infty^\u$ and $\kappa_c^\u$ can be equal to the exact mixed and componentwise condition numbers for most cases. Thus the proposed upper bounds are sharp. In a word, numerical examples validate the effectiveness of theatrical results on the mixed and componentwise condition numbers and their upper bounds.


\begin{landscape}

\begin{table}[H]
\caption{
Comparison of condition numbers with the corresponding relative
errors and the \textsc{Matlab}  ${\rm cond}(\cA)$.
} \label{tab:T1}
\centering
\begin{tabular}{ccccccccccccc}
\hline
$\eta$ & $\delta$ & $\call$ &${\rm cond}(\cA)$& $r_2^\rel$  & $\kappa_1 $ & $\kappa_2$ & $r_\infty^{\rel}$ &$\kappa_\infty^{\rel}$ &$\kappa_\infty^\u$& $r_c^{\rel}$&$\kappa_c^{\rel}$ &$\kappa_c^\u$ \\
\hline
$10^{-3}$ & $10^{-3}$ &
$I$&   1.8019e+06& 1.0131e-09 & 1.4174e+03 & 3.0000e+03 & 1.0131e-09 & 2 & 2.002 & 1.5796e-09 & 2 & 4  \\
 & &  $\call_1$& &9.9040e-10 &  & 1.7321e+06 & 1.5796e-09 & 2 & 4 & 1.5796e-09 & 2  & 4\\
 & & $\call_2$ & &1.0131e-09 &  & 3.0000e+03 & 1.0131e-09 & 2 & 2& 1.5796e-09 & 2  & 2   \\
\hline
$10^{-3}$ & $10^{-6}$  &
$I$ &    1.8019e+12 &2.3251e-09 & 1.4157e+06 & 2.8286e+06 & 2.3251e-09 & 2 & 2.002 & 5.0255e-09 & 2  & 4  \\
&&$\call_1$& &3.0982e-09 &  & 1.6331e+09 & 5.0255e-09 & 2 & 4& 5.0255e-09 & 2  & 4\\
&&$\call_2$ & &2.3251e-09 &  & 2.8286e+06 & 2.3251e-09 & 2 & 2 & 5.0255e-09 & 2  & 2  \\
\hline
$10^{-6}$ & $10^{-3}$ &
$I$&   1.8019e+06 &4.6600e-09 & 1.4166e+03 & 1.0000e+06 & 4.6600e-09 & 2 & 2& 7.3082e-09 & 2  & 4\\
&&$\call_1$& &4.5676e-09 &  & 5.7735e+11 & 7.3082e-09 & 2 & 4 & 7.3082e-09 & 2  & 4\\
&&$L_2$&  &4.6600e-09 & & 1.0000e+06 & 4.6600e-09 & 2 & 2 & 7.3082e-09 & 2  & 2 \\
\hline
$10^{-6}$ & $10^{-6}$ &
$I$&  1.8019e+12&4.1157e-09 & 1.4142e+06 & 3.0000e+06 & 4.1157e-09 & 2 & 2 & 8.2246e-09 & 2  & 4 \\
&&$\call_1$& & 4.9302e-09 &  & 1.7321e+12 & 8.2246e-09 & 2 & 4 & 8.2246e-09 & 2  & 4\\
&&$\call_2$ & &4.1157e-09 &  & 3.0000e+06 & 4.1157e-09 & 2 & 2& 8.2246e-09 & 2  & 2  \\
\hline
\end{tabular}
\end{table}

\end{landscape}

\section{Concluding Remarks}\label{sec:con}

In this paper we studied the perturbation analysis for the least squares  problem with equality constrains.  Condition number expressions for the linear function of the LSE solution were derived through the dual techniques under componentwise perturbations for the input data. Moreover,  sharp upper bounds for mixed and componentwise condition numbers could be estimated efficiently by the Hager-Higham algorithm~\cite{HagerCond84,HighamFortran88,HighamSISC90} when solving LSE using the generalized QR factorization method \cite{AndersonBai92LAA,CoxHigham99,Hammarling87,Paige90QR}. Numerical examples validated the effectiveness of the proposed condition numbers.

\end{document}